\documentclass[a4paper,10pt,psamsfonts]{amsart}

\usepackage[utf8]{inputenc} 

\usepackage{amsmath}
\usepackage{amstext}
\usepackage{amsfonts}
\usepackage{amsthm}
\usepackage{amssymb}

\usepackage{graphicx}

\usepackage{hyperref}
\hypersetup{colorlinks}

\usepackage{color}

\newcommand{\R}{{\mathbb{R}}}

\newcommand{\E}{{\mathbb{E}}}
\newcommand{\N}{{\mathbb{N}}}

\newcommand{\F}{{\mathcal{F}}} 
\renewcommand{\P}{{\mathbb{P}}} 

\newcommand{\diff}[1]{\,\mathrm{d}#1}

\newcommand{\one}{\mathbb{I}}

\newcommand{\triple}{{\vert\kern-0.25ex\vert\kern-0.25ex\vert}}

\theoremstyle{plain}

\newtheorem{definition}{Definition}[section]
\newtheorem{theorem}[definition]{Theorem}
\newtheorem{lemma}[definition]{Lemma}

\newtheorem{prop}[definition]{Proposition}
\newtheorem{assumption}[definition]{Assumption}

\theoremstyle{definition}
\newtheorem{remark}[definition]{Remark}

\begin{document}

\title[Error analysis of randomized Runge-Kutta methods]
{Error analysis of randomized Runge-Kutta methods 
for differential equations with time-irregular coefficients}

\author[R.~Kruse]{Raphael Kruse}
\address{Raphael Kruse\\
Technische Universit\"at Berlin\\
Institut f\"ur Mathematik, Secr. MA 5-3\\
Stra\ss e des 17.~Juni 136\\
DE-10623 Berlin\\
Germany}
\email{kruse@math.tu-berlin.de}

\author[Y.~Wu]{Yue Wu}
\address{Yue Wu\\
Technische Universit\"at Berlin\\
Institut f\"ur Mathematik, Secr. MA 5-3\\
Stra\ss e des 17.~Juni 136\\
DE-10623 Berlin\\
Germany}
\email{wu@math.tu-berlin.de}


\begin{abstract}
  This paper contains an error analysis of two randomized explicit Runge-Kutta
  schemes for ordinary differential equations (ODEs) with time-irregular
  coefficient functions. In particular, the methods are applicable to ODEs of
  Carath\'eodory type, whose coefficient functions are only integrable with
  respect to the time variable but are not assumed to be continuous. A further
  field of application are ODEs with coefficient functions that contain weak
  singularities with respect to the time variable. 

  The main result consists of precise bounds for the discretization error with
  respect to the $L^p(\Omega;\R^d)$-norm. In addition, convergence rates are
  also derived in the almost sure sense. An important ingredient in the
  analysis are corresponding error bounds for the randomized Riemann sum
  quadrature rule. The theoretical results are illustrated through a few
  numerical experiments. 
\end{abstract}

\maketitle

\section{Introduction}
\label{sec:intro}

In this paper we investigate the numerical solution of ordinary differential
equations by randomized one-step methods. More precisely, let $T \in
(0, \infty)$ and let $u 
\colon [0,T] \to \R^d$, $d \in \N$, denote the exact solution to the initial
value problem  
\begin{align}
  \label{eq4:CDE}
  \begin{split}
    \begin{cases}
      \dot{u}(t) &= f(t,u(t)), \quad t \in [0,T],\\
      u(0) &= u_0,
    \end{cases}
  \end{split}
\end{align}
where $u_0 \in \R^d$ is the initial condition. Let us recall that the initial
value problem \eqref{eq4:CDE} is said to be of \emph{Carath\'eodory type} if
the measurable coefficient function $f \colon [0,T] \times \R^d \to \R^d$ is
(locally) integrable with respect to the temporal variable and continuous with
respect to the state variable. If $f$ is additionally assumed to fulfill a
(local) Lipschitz condition with respect to the state variable, then it is
well-known that the initial value problem \eqref{eq4:CDE} admits a unique
(local) solution $u$.
Recall that a
measurable mapping $u \colon [0,T] \to \R^d$ is called a (global)
\emph{solution} to \eqref{eq4:CDE} if $u$ is absolutely continuous and
satisfies  
\begin{align}
  \label{eq4:sol}
  u(t) = u_0 + \int_0^t f(s,u(s)) \diff{s}
\end{align}
for any $t \in [0,T]$. In particular, Equation~\eqref{eq4:sol} comes implicitly
with the assumption that the mapping $[0,T] \ni t \mapsto f(t,u(t)) \in \R^d$
is integrable.  For instance, we refer to \cite[Chap.~I,
Thm~5.3]{hale1980}. 

Our motivation for studying Carath\'eodory type initial value problems stems
from the fact that certain \emph{stochastic differential equations} 
\cite{mao2008, oksendal2003} or \emph{rough differential equations}
\cite{friz2014} that are driven by an additive noise can be transformed into a
problem of the form \eqref{eq4:CDE}. For instance, let $b \colon \R^d \to \R^d$
be Lipschitz continuous and let $v \colon [0,T] \to \R^d$ be the solution to a
rough differential equation of the form
\begin{align*}
  \diff{v(t)} = b(v(t)) \diff{t} + \diff{r(t)}, \quad v(0) = v_0,
\end{align*}
where $r \colon [0,T] \to \R^d$ is a non-smooth but integrable perturbation.
Then, the mapping $u \colon [0,T] \to \R^d$ given by $u(t) := v(t) - r(t)$, $t
\in [0,T]$, solves an ODE of the form \eqref{eq4:CDE} with $f(t,x) := b(r(t) +
x)$ for each $(t,x) \in [0,T] \times \R^d$. Depending on the smoothness of the
perturbation $r$, the resulting mapping $f$ is then  often only integrable or
H\"older continuous with exponent $\gamma \in (0,1]$ with respect to the
temporal variable $t$.

Due to the low regularity of the coefficient function $f$ the numerical
approximation of the solution $u$ to Carath\'eodory type differential equations
is a challenging task. Indeed, it can be shown that any deterministic numerical
one-step method is in general divergent if it only uses finitely many point
evaluations of $f$. This is easily seen from 
a simple adaptation of arguments presented in \cite[Kap.~2.3]{mueller2012}.
We discuss this aspect in more depth in Section~\ref{subsec:divergence} below.

If the coefficient function $f$ enjoys slightly more regularity with respect
to the temporal variable, say H\"older continuous with some exponent $\gamma
\in (0,1]$ (compare with Assumption~\ref{as5:f} further below),
then classical deterministic numerical algorithms such as 
Runge-Kutta methods or linear multi-step methods become applicable and will
converge to the exact solution. However, since $f$ is not
assumed to be differentiable we cannot expect high order of convergence from
these schemes. In fact, in \cite{kacewicz1987} it is shown that if $\gamma = 1$
then the minimum error of any deterministic method depending only on $N \in
\N$ point evaluations of the coefficient function $f$ is of order
$\mathcal{O}(N^{-1})$. Similarly, for arbitrary values of
$\gamma \in (0,1)$ the minimum error among all deterministic algorithms that
use at most $N \in \N$ point evaluations of $f$ decays only with order
$\mathcal{O}(N^{-\gamma})$, see \cite{heinrich2008}. Therefore, especially in
the case of small values for $\gamma$, deterministic methods may still be
considered as impracticable.

For these reasons it is necessary to extend the class of considered numerical
algorithms. For instance, the method could additionally make use of linear
functionals of the coefficient function $f$, say  
integrals, instead of mere point evaluations. However, it is often not clear
how to implement these 
methods if $f$ is not (piecewise) continuous. Here, we therefore follow a
different path that considers \emph{randomized numerical methods}. The
prototype of this class of numerical algorithms is the classical Monte Carlo
method, which converges already under an integrability condition.

In the literature, several randomized numerical methods have been
developed for the specific initial value problem \eqref{eq4:CDE} under a
variety of mild regularity assumptions. For instance, we mention the results in 
\cite{daun2011,heinrich2008,jentzen2009, kacewicz2006,stengle1990,stengle1995}
and the references therein. These randomized methods are usually found to be
superior over corresponding deterministic methods in the sense that the
resulting discretization error decays already with order
$\mathcal{O}(N^{-\gamma - \frac{1}{2}})$ under the \emph{same} smoothness
assumptions as sketched above. Let us also mention that a further application
of randomized methods to initial value problems in Banach spaces is found in
\cite{daun2014, heinrich2013}, while the approximation of stochastic
ODEs by a randomized Euler-Maruyama method is considered in
\cite{przybylowicz2014}. In \cite{coulibaly1999} a related family of
quasi-randomized methods is studied. 

In this paper we present a precise error analysis for two randomized
Runge-Kutta methods that are applicable to the numerical solution of ODEs with
time-irregular coefficient functions. The purpose is to prove convergence
of the two methods with an order of at least $\frac{1}{2}$ with respect to the
$L^p(\Omega;\R^d)$-norm under very mild conditions on the coefficient function
$f$. Hereby we relax several conditions on $f$
often found in the literature. In particular, we do not assume that the
coefficient function is (locally) bounded which allows to treat functions $f$
with a weak singularity of the form $|t - t_0|^{-\frac{1}{p} + \epsilon}$, $p
\in [2,\infty)$, $\epsilon \in (0,\infty)$, $t, t_0 \in [0,T]$. In addition,
we also estimate the order of convergence in the almost sure sense. The precise
conditions on the coefficient function are stated in Assumption~\ref{as4:f} and
Assumption~\ref{as5:f}.

We now introduce the two randomized Runge-Kutta methods in more detail. If the
reader is not familiar with standard notations and concepts in probability we
suggest to first consult Section~\ref{sec:notation}.  

Let $(\tau_j)_{j \in \N}$
be a sequence of independent and $\mathcal{U}(0,1)$-distributed random
variables on a probability space $(\Omega, \F, \P)$. 
Then, for any step size $h \in (0,1)$ we define $N_h \in \N$ to be the integer
determined by $N_h h \le T < (N_h + 1) h$. Set $t_j = j h$ for every $j \in
\N\cup\{0\}$. The first numerical approximation $(U^j)_{j \in \{0,\ldots,N_h\}}$
of $u$ considered in this paper is determined by setting $U^0 = u_0$ and by the recursion 
\begin{align}
  \label{eq5:NumMeth1}
  \begin{split}
    U^j &= U^{j-1} + h f(t_{j-1} + \tau_j h, U^{j-1}),
  \end{split}
\end{align}
for all $j \in \{1,\ldots, N_h\}$. This method is usually termed
\emph{randomized Euler method} and it is a particular case of a Runge-Kutta Monte
Carlo method studied in \cite{jentzen2009, stengle1990, stengle1995}. Let us
emphasize that, as it is customary for Monte Carlo methods, the result of the
numerical scheme is a discrete time stochastic process defined on the same
probability space as the random input $(\tau_j)_{j \in \N}$.

The second randomized Runge-Kutta method $(V^j)_{j
\in \{0,\ldots,N_h\}}$ is determined by setting $V^0 = u_0$ and by the
recursion 
\begin{align}
  \label{eq5:NumMeth2}
  \begin{split}
    V^j_\tau &= V^{j-1} + h \tau_j f(t_{j-1}, V^{j-1}),\\
    V^j &= V^{j-1} + h f(t_{j-1} + \tau_j h, V^j_\tau),
  \end{split}
\end{align}
for all $j \in \{1,\ldots, N_h\}$. This scheme is a member of a family of
methods that has been introduced in \cite{daun2011, heinrich2008}.  

The two methods \eqref{eq5:NumMeth1} and \eqref{eq5:NumMeth2}
can indeed be interpreted as \emph{randomized Runge-Kutta methods}. In fact, 
in the $j$-th step of \eqref{eq5:NumMeth1} and \eqref{eq5:NumMeth2} we randomly
choose one particular Runge-Kutta method from the families with Butcher tableaux
\begin{align}
  \label{eq5:randRK}
    &\begin{array}{c|c}
      \theta & 0 \\
      \hline
       & 1
    \end{array},
    && \text{ or }
    & \quad
    \begin{array}{c|cc}
      0 & 0 & 0 \\
      \theta & \theta & 0 \\
      \hline
      & 0 & 1
    \end{array},
\end{align}
respectively, where the value of the parameter $\theta \in [0,1]$ is determined
by the outcome of the random variable $\tau_j$. For more details on Runge-Kutta
methods and their Butcher-tableaux \cite{butcher1964} we refer to standard
references, for example \cite{butcher2008, grigorieff1972, hairer1993}.

The remainder of this paper is organized as follows. In
Section~\ref{sec:notation} we introduce our notation and recall some
prerequisites from probability that are needed later. In
Section~\ref{sec:randRie} we state and prove precise error estimates for
the \emph{randomized Riemann sum quadrature rule}, which are an important
ingredient in our error analysis for the randomized Runge-Kutta methods
\eqref{eq5:NumMeth1} and \eqref{eq5:NumMeth2}. Randomized quadrature rules are
well-known to the literature, see \cite{haber1966, haber1967}. 
However, this is apparently the first time they are
applied in the error analysis of randomized Runge-Kutta methods.

Section~\ref{sec:CDE} contains the first main result of this paper. Here we
prove that the randomized Euler method \eqref{eq5:NumMeth1} converges to the
exact solution of a Carath\'eodory type ODE \eqref{eq4:CDE} with order
$\frac{1}{2}$ with respect to the norm in $L^p(\Omega;\R^d)$. See
Assumption~\ref{as4:f} for a precise statement of the conditions on the
coefficient function $f$. In addition we also derive the order of convergence
in the almost sure sense, hereby generalizing results from \cite{jentzen2009}
to unbounded coefficient functions. Note that the computationally more expensive
method \eqref{eq5:NumMeth2} does not offer any additional advantages in terms
of convergence speed in case of possibly discontinuous coefficient functions.
We therefore omit an error analysis in this situation. 

In Section~\ref{sec:ODE} we then consider the classical ODE setting with
a H\"older continuous coefficient function $f$. We determine the order of
convergence of the two numerical methods in dependence of the H\"older exponent
$\gamma$ and with respect to the $L^p(\Omega;\R^d)$-norm. We see that the
randomized Runge-Kutta method 
\eqref{eq5:NumMeth2} is superior to the randomized Euler method
\eqref{eq5:NumMeth1} if $\gamma \in (\frac{1}{2},1]$. These results generalize
the error analysis from \cite{daun2011, heinrich2008} to the case $p >
2$. Since they are based on the $L^p$-convergence result, we believe that our
almost sure convergence rates are new to the literature as well. Lastly, 
we present several numerical experiments in the final section.

\subsection{Divergence of deterministic algorithms}
\label{subsec:divergence}

As announced in the introduction let us briefly follow a line of arguments from
\cite[Kap.~2.3]{mueller2012}. Our aim is to give a sketch of proof that all
deterministic algorithms that only use point evaluations of $f$ will in general
diverge if applied to Carath\'eodory type ODEs. 

To this end, let $T = 1$, $d = 1$, and $u_0 = 0$ and consider the
problem \eqref{eq4:CDE} with the coefficient function $f_1(t,x) \equiv 1$ for
all $t \in [0,T]$ and $x \in \R$. Clearly, in this case the exact solution
$u_1$ satisfies $u_1(1) = 1$. If we apply an arbitrary but fixed deterministic
algorithm for the approximation of $u_1(1)$ with $N \in \N$ evaluations of
$f_1$ it will return an approximation $U_N \in \R$. Let us assume that for the
computation of $U_N$ the deterministic algorithm evaluated the coefficient
function $f_1$ at the points $(t_i^N, x_i^N) \in [0,T] \times \R$, $i =
1,\ldots,N$, in the extended phase space. For each number $N \in \N$ define now
the set $B_N = \cup_{ i = 1}^N \{ t_i^N \} \subset [0,T]$ as well as $B :=
\cup_{n \in \N} B_n \subset [0,T]$.  

Then, we consider a further initial value problem with the same initial
condition $u_0$ but with the coefficient function $f_2(t,x) :=
\one_B(t)$ for all $(t,x) \in [0,T] \times \R$. Obviously, the mapping $f_2$ is
also measurable and bounded. Since $f_2$ does not depend on the state variable,
it also fits into the framework of Carath\'eodory type ODEs. In fact, the mapping
$f_2$ fulfills all conditions of Assumption~\ref{as4:f} further below.
Because the set $B$ has Lebesgue measure zero the exact solution $u_2$
satisfies $u_2(1) = 0$ in this case. However, if we now apply the same
deterministic algorithm as above, it cannot distinguish between $f_1$ and $f_2$
and it will return the same numerical approximation $U_N \in \R$. Since this is
true for any $N \in \N$ and since $u_1(1) = 1 \neq 0 = u_2(1)$ the
deterministic algorithm will not converge to the exact solution of at least one
of the problems.

\section{Preliminaries}
\label{sec:notation}

In this section we collect a few important results and
inequalities in particular from probability, which are needed later. But first
we fix some notation and terminology that is frequently used throughout this
paper. 

As usual we denote by $\N$ the set of all positive integers and $\N_0 = \N \cup
\{0\}$, while $\R$ denotes the set of all real numbers. By $| \cdot |$ we
denote the standard norm on the 
Euclidean space $\R^d$, $d \in \N$. Further, for every $\gamma \in (0,1]$ we
denote by $\mathcal{C}^{\gamma}([0,T]) := \mathcal{C}^{\gamma}([0,T];\R^d)$ the
set of all $\gamma$-H\"older continuous mappings $g \colon [0,T] \to \R^d$.
Note that the space $\mathcal{C}^{\gamma}([0,T])$ becomes a Banach space if
endowed with the H\"older norm 
\begin{align*}
  \| g \|_{\mathcal{C}^{\gamma}([0,T])} = \sup_{t \in [0,T]} | g(t) | 
  + \sup_{\substack{t,s \in [0,T]\\ t \neq s}} \frac{| g(t) -
  g(s)|}{|t-s|^\gamma}. 
\end{align*}
In particular, it then holds true that
\begin{align*}
  | g(t) - g(s) | \le \| g \|_{\mathcal{C}^{\gamma}([0,T])} |t-s|^\gamma, \quad
  \text{ for all } t,s \in [0,T].
\end{align*}
The next inequality is a
useful tool to bound the error of a numerical approximation. For a proof and 
more general variants see for instance Proposition 4.1 in \cite{emmrich1999}.
 
\begin{lemma}[Discrete Gronwall's inequality]
  \label{lem:Gronwall} Consider two nonnegative sequences $(u_n)_{n\in \N}$
  and $(w_n)_{n\in \N}$ which for some given $a \in [0,\infty)$ satisfy 
  $$u_n \leq a + \sum_{j=1}^{n-1} w_j u_j,\quad \text{ for all }  n \in \N,$$
  then for all $n \in \N$ it also holds true that
  $$u_n\leq a \exp\Big(\sum_{j=1}^{n-1} w_j \Big).$$
\end{lemma}

For the introduction and the error analysis of Monte Carlo methods, we also
require some fundamental concepts from probability and stochastic analysis.
For a general introduction readers are referred to standard monographs
on this topic, for instance \cite{kallenberg2002, klenke2014, oksendal2003}. 
For the measure theoretical background see \cite{bauer2001, cohn2013}.

First let us recall that a \emph{probability space} $(\Omega,\mathcal{F},\P)$
consists of a measurable space $(\Omega,\mathcal{F})$ endowed with a finite
measure $\P$ satisfying $\P(\Omega) = 1$. The value $\P(A) \in [0,1]$ is
interpreted as the \emph{probability} of the \emph{event} $A \in \F$.
A mapping $X \colon \Omega \to \R^d$
is called a \emph{random variable} if $X$ is $\F /
\mathcal{B}(\R^d)$-measurable, where $\mathcal{B}(\R^d)$ denotes the
Borel-$\sigma$-algebra generated by the set of all open subsets of $\R^d$.
More precisely, it holds true that
\begin{align*}
  X^{-1}(B)= \big\{ \omega \in \Omega\, : \, X(\omega)\in B \big\} \in
  \mathcal{F} 
\end{align*}
for all $B \in \mathcal{B}(\R^d)$. Every random variable induces a probability
measure on its image space. In fact, the measure $\mu_X \colon
\mathcal{B}(\R^d) \to [0,1]$ given by $\mu_X(B)=\P(X^{-1}(B))$ for all
$B \in \mathcal{B}(\R^d)$ is a probability measure on the measurable space
$(\R^d, \mathcal{B}(\R^d))$. Usually, $\mu_X$ is called the \emph{distribution}
of $X$. 

In this paper, we frequently encounter a family of
$\mathcal{U}(a,b)$-distributed random variables $(\tau_j)_{j \in \N}$. This
means that for each $j \in \N$ the real-valued mapping $\tau_j \colon \Omega
\to \R$ is a random variable which is \emph{uniformly distributed} on the
interval $(a,b)$ with $a,b \in \R$, $a < b$. In particular, the distribution
$\mu_{\tau_j}$ of $\tau_j$ is given by $\mu_{\tau_j}(A) = \frac{1}{(b-a)}
\lambda\big( A \cap (a,b) \big)$, where $\lambda$ denotes the Lebesgue measure
on the real line. 

Next, let us recall that a random variable $X \colon \Omega \to \R^d$ is
called \emph{integrable} if $\int_{\Omega}|X(\omega)| \diff{\P(\omega)}<\infty$. 
Then, the \emph{expectation} of $X$ is defined as
$$\E[X]:=\int_{\Omega}X(\omega)\diff{\P(\omega)} = \int_{\R^d} x
\diff{\mu_{X}(x)}.$$ 
Moreover, we write $X \in L^p(\Omega;\R^d)$ with $p \in [1,\infty)$ if 
$\int_{\Omega}|X(\omega)|^p \diff{\P(\omega)}<\infty$. In addition, the set
$L^p(\Omega;\R^d)$ becomes a Banach space if we identify all random
variables which only differ on a set of measure zero (i.e. probability zero)
and if we endow $L^p(\Omega;\R^d)$ with the norm
\begin{align*}
  \| X \|_{L^p(\Omega;\R^d)} = \big( \E \big[ |X|^p \big]
  \big)^{\frac{1}{p}} = \Big( \int_{\Omega}|X(\omega)|^p \diff{\P(\omega)}
  \Big)^{\frac{1}{p}}.
\end{align*}
This definition coincides with the definition of the standard spaces
$L^p([0,T];\R^d)$ of $p$-fold Lebesgue-integrable measurable functions. If $X
\in L^p(\Omega;\R^d)$ for some $p \in [1,\infty)$ the \emph{Chebyshev
inequality} yields for all $\lambda \in (0,\infty)$ 
\begin{align}
  \label{eq:Cheb}
  \P \big( \{ \omega \in \Omega \, : \,  |X(\omega)| > \lambda \}
  \big) \le \| X \|^p_{L^p(\Omega;\R^d)} \lambda^{-p}.
\end{align}

Further, we say that a family of $\R^d$-valued random variables $(X_n)_{n \in \N}$ is 
\emph{independent} if for any finite subset $M \subset \N$ and for arbitrary events
$(A_m)_{m \in M} \subset \mathcal{B}(\R^d)$ we have the multiplication rule
\begin{align*}
  \P \Big( \bigcap_{m \in M} \{ \omega \in \Omega\, : \, X_m(\omega) \in A_m \}
  \Big) = \prod_{m \in M} \P \big( \{ \omega \in \Omega\, : \, X_m(\omega) \in
  A_m \} \big).
\end{align*}
On the level of the distributions of $(X_m)_{m \in \N}$ this basically means
that the joint distribution of each finite subfamily $(X_m)_{m \in M}$ is equal
to the product measure of the single distributions. From this we directly get
the multiplication rule for the expectation
\begin{align}
  \label{eq:prod_ind}
  \E\Big[ \prod_{m \in M} X_m \Big] = \prod_{m \in M} \E \big[ X_m \big],
\end{align}
provided $X_m$ is integrable for each $m \in M$.

If we interpret the index as a time parameter, we say that a 
family of $\R^d$-valued random variables $(X_m)_{m \in \N}$ is a discrete time
\emph{stochastic process}. A very important class of stochastic processes are
\emph{martingales}. Without stating a precise definition of martingales
it suffices for the understanding of this paper to be aware of the fact that
if $(X_m)_{m \in \N}$ is an independent family of integrable random variables
satisfying $\E [X_m] = 0$ for each $m \in \N$, then the stochastic process
defined by the partial sums  
\begin{align*}
  S_n := \sum_{m = 1}^n X_m, \quad n \in \N, 
\end{align*}
is a discrete time martingale. This enables us to apply
powerful inequalities for martingales, such as the following discrete time
version of the Burkholder–Davis–Gundy inequality, see \cite{burkholder1966}.

\begin{theorem}[Burkholder–Davis–Gundy inequality]
  \label{th:BDG}
  For each $p \in (1,\infty)$ there exist positive constants $c_p$ and $C_p$
  such that for every discrete time martingale $(X_n)_{n \in \N}$ and
  for every $n \in \N$ we have 
  $$c_p \| [X]_{n}^{{1}/{2}} \|_{L^p(\Omega;\R)}
  \leq \big\| \max_{j\in \{1,\ldots,n\} } |X_j| \big\|_{L^p(\Omega;\R)} 
  \leq C_p \big\| [X]_{n}^{{1}/{2}} \big\|_{L^p(\Omega;\R)},$$
  where $[X]_n = |X_1|^2 + \sum_{k=2}^{n} |X_{k}-X_{k-1}|^2$ denotes the \emph{quadratic
  variation} of $(X_n)_{n \in \N}$ up to $n$.
\end{theorem}

Another well-known lemma considers the limiting behaviour of sequences of sets
under probability measure (see Theorem 2.7 in \cite{klenke2014}).
\begin{lemma}[Borel-Cantelli Lemma]
  \label{lem:BC}
  If $A_1,A_2,\cdots\in \mathcal{F}$ and $\sum_{n=1}^{\infty}\P(A_n) < \infty$,
  then $\P(\limsup_{n \to \infty} A_n)=0$, where
  \begin{align*}
    \limsup_{n \to \infty} A_n = \bigcap_{n = 1}^\infty \bigcup_{i = n}^\infty
    A_i = \big\{ \omega \in \Omega \, : \, \omega \in A_i \text{ for infinitely
    many $i \in \N$ } \big\}.    
  \end{align*}
\end{lemma}

\section{Error estimates for randomized Riemann sums}
\label{sec:randRie}

In this section we give precise error estimates for a randomized Riemann sum
quadrature rule for integrals whose integrands have various degrees of
smoothness. Randomized quadrature rules have been first introduced in
\cite{haber1966, haber1967}. Usually, they consist of a randomized version
of classical deterministic quadrature rules and are known to offer advantages 
if the integrand is not smooth. However, in contrast to most Monte Carlo
methods, randomized quadrature rules still suffer from the curse of
dimensionality in the same way as their deterministic counter-parts. The main
field of application therefore is the numerical approximation of integrals with
a non-smooth integrand over a low-dimensional domain. See also
\cite[Sec.~6.4.5]{evans2000} or \cite[Sec.~5.5]{mueller2012} for further
details.

As in the introduction, for any step size $h \in (0,1)$ we define $N_h
\in \N$ to be the integer determined by $N_h h \le T < (N_h + 1) h$.
Let us recall that for every measurable function $g \colon [0,T]
\to \R^d$ with $\|g\|_{L^p([0,T];\R^d)} < \infty$ for some $p \in [2,\infty)$
the \emph{randomized Riemann sum approximation} $Q_{\tau,h}^n[g]$ of
$\int_0^{t_n} g(s) \diff{s}$ with step size $h \in (0,1)$ is given by 
\begin{align}
  \label{eq3:RandRie}
  Q_{\tau,h}^n[g] := h \sum_{j = 1}^n g(t_{j-1} + h \tau_j),\quad n \in
  \{1,\ldots,N_h\},
\end{align}
where $t_j = jh$ and $(\tau_j)_{j \in \N}$ is an independent family of
$\mathcal{U}(0,1)$-distributed random variables on a probability space
$(\Omega,\F,\P)$. 

The first theorem contains an error estimate with respect to
the $L^p(\Omega;\R^d)$-norm. Further below, we also study the almost sure
convergence of $Q_{\tau,h}^n[g]$.

\begin{theorem}[$L^p$-error estimate]
  \label{th4:randRiemann}
  Let $g \colon [0,T] \to \R^d$ be a measurable mapping satisfying $\| g
  \|_{L^p([0,T];\R^d)} < \infty$ for some $p \in [2, \infty)$. 
  Then, for every $h \in (0,1)$ and $n \in \{1,\ldots,N_h\}$
  the randomized Riemann sum $Q_{\tau,h}^n[g] \in
  L^p(\Omega;\R^d)$ is an unbiased estimator for the 
  integral $\int_0^{t_n} g(s) \diff{s}$, i.e.,
  $\E[Q_{\tau,h}^n[g]]=\int_0^{t_n} g(s) \diff{s}$. Further, for all $h \in
  (0,1)$ we have
  \begin{align}
    \label{eq4:errRie1}
    \Big\| \max_{n \in \{1,\ldots,N_h\}}
    \Big| \int_0^{t_n} g(s) \diff{s} - Q_{\tau,h}^n[g] \Big|
    \, \Big\|_{L^p(\Omega;\R)}
    \le 2 C_p T^{\frac{p-2}{2p}} \|g \|_{L^p([0,T];\R^d)} h^{\frac{1}{2}}.
  \end{align}
  In addition, if the mapping $g$ is $\gamma$-H\"older continuous for some
  $\gamma \in (0,1]$, then for all $h \in (0,1)$ we have  
  \begin{align}
    \label{eq4:errRie2}
    \Big\| \max_{n \in \{1,\ldots,N_h\}} \Big| \int_0^{t_n} g(s) \diff{s} -
    Q_{\tau,h}^n[g] \Big| \, \Big\|_{L^p(\Omega;\R)}
    \le C_p \sqrt{T} \| g \|_{\mathcal{C}^\gamma([0,T])} h^{\frac{1}{2} +
    \gamma}. 
  \end{align}
\end{theorem}

\begin{proof}
  First, due to $\| g \|_{L^p([0,T];\R^d)} < \infty$ and $\tau_j \sim
  \mathcal{U}(0,1)$ it follows
  \begin{align}
    \label{eq4:est_g}
    h \big\| g(t_{j-1} + h \tau_j) \big\|_{L^p(\Omega;\R^d)}^p 
    = \int_{t_{j-1}}^{t_j} | g(s) |^p \diff{s} < \infty.
  \end{align}
  Hence, $Q_{\tau,h}^n[g] \in L^p(\Omega;\R^d)$. After taking the expected
  value of $Q_{\tau,h}^n[g]$ we get 
  \begin{align*}
    \E \big[ Q_{\tau,h}^n[g] \big] &= h \sum_{j = 1}^n \E\big[ g(t_{j-1} +
    h \tau_j) \big] = h  \sum_{j = 1}^n \int_0^1 g(t_{j-1} +
    s h) \diff{s} \\
    &= \sum_{j = 1}^n \int_{t_{j-1}}^{t_j} g(s) \diff{s} =
    \int_{0}^{t_n} g(s) \diff{s}. 
  \end{align*} 
  Thus, the randomized Riemann sum is an unbiased estimator for
  $\int_0^{t_n} g(s) \diff{s}$. Further, by linearity of the integral we obtain
  for the error
  \begin{align*}
    E^n := \int_0^{t_n} g(s) \diff{s} - Q_{\tau,h}^n[g] 
    &= \sum_{j = 1}^n \int_{t_{j-1}}^{t_j} \big( g(s) - g(t_{j-1} + h \tau_j)
    \big) \diff{s}. 
  \end{align*}
  Now, as above it follows that each summand is a centered random variable,
  that is  
  \begin{align*}
    \E \Big[ \int_{t_{j-1}}^{t_j} \big( g(s) - g(t_{j-1} + h \tau_j)
    \big) \diff{s} \Big] = 0
  \end{align*}
  for every $j \in \N$. Moreover, the summands are mutually independent
  due to the independence of $(\tau_j)_{j \in \N}$. In addition, we also obtain
  from \eqref{eq4:est_g} that
  \begin{align*}
    &\Big\| \int_{t_{j-1}}^{t_j} \big( g(s) - g(t_{j-1} + h
    \tau_j) \big) \diff{s} \Big\|_{L^p(\Omega;\R^d)}\\ 
    &\quad \le \int_{t_{j-1}}^{t_j} | g(s) | \diff{s} 
    + h \big\| g(t_{j-1} + h \tau_j) \big\|_{L^p(\Omega;\R^d)} < \infty.
  \end{align*}
  Therefore,
  $(E^n)_{n \in \{1,\ldots,N_h\}}$ is a discrete time
  $L^p$-martingale. Thus, we can apply the Burkholder-Davis-Gundy inequality
  from Theorem~\ref{th:BDG} and obtain  
  \begin{align*}
    \big\| \max_{n \in \{1,\ldots,N_h\}} | E^n| \big\|_{L^p(\Omega;\R)}
    \le C_p \big\| [E]^{\frac{1}{2}}_{N_h} \big\|_{L^p(\Omega;\R)}. 
  \end{align*}
  After inserting the quadratic variation $[E]_{N_h}$ we arrive at
  \begin{align}
    \label{eq4:est_E}
    \begin{split}
      &\big\| \max_{n \in \{1,\ldots,N_h\}} | E^n|
      \big\|_{L^p(\Omega;\R)}\\
      &\quad \le C_p \Big\| \Big( \sum_{j = 1}^{N_h}
      \Big| \int_{t_{j-1}}^{t_j} \big( g(s) -
      g(t_{j-1} +  h \tau_j) \big) \diff{s}  \Big|^2 \Big)^{\frac{1}{2}} 
      \Big\|_{L^p(\Omega;\R)}\\
      &\quad = C_p \Big\| \sum_{j = 1}^{N_h} \Big| \int_{t_{j-1}}^{t_j} \big(
      g(s) - g(t_{j-1} +  h \tau_j) \big) \diff{s} \Big|^2 \,
      \Big\|_{L^{\frac{p}{2}}(\Omega;\R)}^{\frac{1}{2}}\\
      &\quad \le C_p \Big( \sum_{j = 1}^{N_h} \Big\| \int_{t_{j-1}}^{t_j}
      \big( g(s) - g(t_{j-1} +  h \tau_j)
      \big) \diff{s}  \Big\|_{L^{p}(\Omega;\R^d)}^2 
      \Big)^{\frac{1}{2}}.
    \end{split}
  \end{align}
  Now, by an application of the triangle inequality we get
  \begin{align*}
    &\big\| \max_{n \in \{1,\ldots,N_h\}} | E^n|
    \big\|_{L^p(\Omega;\R)}\\
    &\quad \le C_p \Big( \sum_{j = 1}^{N_h} \Big| \int_{t_{j-1}}^{t_j}
    g(s) \diff{s}\Big|^2 \Big)^{\frac{1}{2}} + C_p 
    \Big( \sum_{j = 1}^{N_h} h^2 \big\|
    g(t_{j-1} +  h \tau_j) \big\|_{L^{p}(\Omega;\R^d)}^2
    \Big)^{\frac{1}{2}}.
  \end{align*}
  The first term is then bounded by
  \begin{align}
    \label{eq4:term1}
    \begin{split}
      \Big( \sum_{j = 1}^{N_h} \Big| \int_{t_{j-1}}^{t_j}
      g(s) \diff{s}\Big|^2 \Big)^{\frac{1}{2}} 
      &\le h^{\frac{1}{2}}
      \Big( \sum_{j = 1}^{N_h} \int_{t_{j-1}}^{t_j}
      \big| g(s) \big|^2 \diff{s} \Big)^{\frac{1}{2}} \le \| g
      \|_{L^2([0,T];\R^d)} h^{\frac{1}{2}}\\
      &\le T^{\frac{p-2}{2p}}
      \| g  \|_{L^p([0,T];\R^d)} h^{\frac{1}{2}},
    \end{split}
  \end{align}
  since $\| g \|_{L^2([0,T];\R^d)} \le 
  T^{\frac{p-2}{2p}} \| g \|_{L^p([0,T];\R^d)}$ by H\"older's inequality.

  If $p = 2$ we directly obtain the same bound for the second term by making
  use of \eqref{eq4:est_g}. If $p \in (2,\infty)$
  we first apply H\"older's inequality with
  exponents $\rho = \frac{p}{2} \in (1,\infty)$ and $\rho' =
  \frac{p}{p - 2} \in (1,\infty)$. This yields
  \begin{align}
    \label{eq4:term2}
    \begin{split}
      & \sum_{j = 1}^{N_h} h^{2 - \frac{2}{p}} \cdot h^{\frac{2}{p}} \big\|
      g(t_{j-1} +  h \tau_j) \big\|_{L^{p}(\Omega;\R^d)}^2 \\
      &\quad \le \Big( \sum_{j = 1}^{N_h} h^{\rho'(2 - \frac{2}{p})}
      \Big)^{\frac{1}{\rho'}} \Big( \sum_{j = 1}^{N_h} h  \big\|
      g(t_{j-1} +  h \tau_j) \big\|_{L^{p}(\Omega;\R^d)}^p
      \Big)^{\frac{2}{p}}\\
      &\quad \le T^{\frac{1}{\rho'}} h^{2( 1 - \frac{1}{p}) - \frac{1}{\rho'}}      
      \big\| g \big\|_{L^p([0,T];\R^d)}^2, 
    \end{split}
  \end{align}
  due to \eqref{eq4:est_g}. Altogether, since $T^{\frac{1}{2 \rho'}} =
  T^{\frac{p - 2}{2 p}}$ and after noting that $h^{2(
  1 - \frac{1}{p}) - \frac{1}{\rho'}} = h$ we derive from
  \eqref{eq4:est_E}, \eqref{eq4:term1} and \eqref{eq4:term2} that
  \begin{align*}
     &\big\| \max_{n \in \{1,\ldots,N_h\}} | E^n|
    \big\|_{L^p(\Omega;\R)} \le 2 C_p T^{\frac{p-2}{2p}} \| g
    \|_{L^p([0,T];\R^d)} h^{\frac{1}{2}}.
  \end{align*}
  This completes the proof of \eqref{eq4:errRie1}.

  Next, if in addition $g \in \mathcal{C}^\gamma([0,T])$, then
  we can improve the estimate in \eqref{eq4:est_E} by 
  \begin{align*}
    &\Big\| \int_{t_{j-1}}^{t_j} \big( g(s) - g(t_{j-1} +  h \tau_j)
    \big) \diff{s}  \Big\|_{L^{p}(\Omega;\R^d)}\\
    &\quad \le \int_{t_{j-1}}^{t_j} \big\| g(s) - g(t_{j-1} +  h \tau_j)
    \big\|_{L^{p}(\Omega;\R^d)} \diff{s}
    \le \| g \|_{\mathcal{C}^\gamma([0,T])} h^{1 + \gamma}.  
  \end{align*}
  Thus, inserting this into \eqref{eq4:est_E} gives
  \begin{align*}
    \big\| \max_{n \in \{0,1,\ldots,N_h\}} | E^n| \big\|_{L^p(\Omega;\R)}
    &\le C_p \Big( \sum_{j = 1}^{N_h} \| g \|_{\mathcal{C}^\gamma([0,T])}^2
    h^{2(1 + \gamma)} \Big)^{\frac{1}{2}}\\
    &\le C_p T^{\frac{1}{2}} \| g \|_{\mathcal{C}^\gamma([0,T])}
    h^{\frac{1}{2} + \gamma}.
  \end{align*}
  This completes the proof of \eqref{eq4:errRie2}.
\end{proof}

Error estimates with respect to the $L^p$-norm are sometimes unsatisfactory,
since they allow for the possibility that single realizations of the randomized
Riemann sum may differ significantly from its expected value. But in
practice often just one realization of the estimator is computed. To some
extent this is justified by the next theorem. This indicates that already on
the level of single ``typical'' realizations of $Q_{\tau,h}^n[g](\omega)$ we
observe convergence provided the step size $h$ is sufficiently small.
However, depending on the value of $p \in (2,\infty)$ the order of convergence
may be significantly reduced. 

\begin{theorem}[Almost sure convergence]
  \label{th3:Rie_pathwise}
  Let $g \colon [0,T] \to \R^d$ be a given measurable mapping with $\| g
  \|_{L^p([0,T];\R^d)} < \infty$ for some $p \in (2, \infty)$.   
  Let $(h_m)_{m \in \N} \subset (0,1)$ be an arbitrary sequence of step sizes
  with $\sum_{m = 1}^\infty h_m < \infty$. Then, there exist a nonnegative
  random variable $m_0 \colon \Omega \to \N_0$ and a measurable set
  $A \in \F$ with $\P(A) = 1$ such that for all $\omega \in A$ and $m \ge
  m_0(\omega)$ we have
  \begin{align}
    \label{eq4:errRie3}
    \max_{n \in \{0,1,\ldots,N_{h_m}\}} \Big| \int_0^{t_n} g(s) \diff{s} -
     Q_{\tau,h_m}^n[g](\omega) \Big| \le h_m^{\frac{1}{2} - \frac{1}{p}}.
  \end{align}
  Moreover, if in addition $g \in \mathcal{C}^\gamma([0,T])$ 
  for some $\gamma \in (0,1]$, then for every $\epsilon \in (0,\frac{1}{2})$
  there exist a nonnegative random variable $m_0^\epsilon \colon \Omega \to
  \N_0$ and a measurable set $A_\epsilon \in \F$ with $\P(A_\epsilon) = 1$ such
  that for all $\omega \in A_\epsilon$ and $m \ge m_0^\epsilon(\omega)$ we have
  \begin{align}
    \label{eq4:errRie4}
    \max_{n \in \{0,1,\ldots,N_{h_m}\}} \Big| \int_0^{t_n} g(s) \diff{s} -
    Q_{\tau,h_m}^n[g](\omega) \Big| \le h_m^{\frac{1}{2} + \gamma - \epsilon}.
  \end{align}
\end{theorem}

For the proof we need the following result, which is a simple consequence of the
Borel-Cantelli lemma. It is a version of \cite[Lemma~2.1]{kloeden2007}, which in
turn is based on a technique developed in \cite{gyoengy1998}.

\begin{lemma}
  \label{lem3:pathwise}
  Let $p \in [1,\infty)$ and $\rho \in (\frac{1}{p}, \infty)$ be given.
  Consider an arbitrary sequence of step sizes $(h_m)_{m \in \N} \subset (0,1)$
  with $\sum_{m = 1}^\infty h_m < \infty$. Then, for every sequence $(X_m)_{m
  \in \N} \subset L^p(\Omega;\R^d)$ satisfying
  \begin{align*}
    \sup_{m \in \N} h_m^{-\rho} \| X_m \|_{L^p(\Omega;\R^d)} < \infty,
  \end{align*}
  there exist a nonnegative random variable $m_0 \colon \Omega \to \N_0$ and a 
  measurable set $A \in \F$ with $\P(A) = 1$ such that 
  for every $\omega \in A$ and $m \ge m_0(\omega)$ it holds true that
  \begin{align*}
    | X_m(\omega) | \le h_m^{\rho - \frac{1}{p}}.  
  \end{align*}
\end{lemma}

\begin{proof}
  For each $m \in \N$ consider the event
  \begin{align*}
    A_m := \big\{ \omega \in \Omega \, : \, | X_m(\omega) | > h_m^{\rho -
    \frac{1}{p}} \big\} \in \F.
  \end{align*}
  Then, by the Chebyshev inequality \eqref{eq:Cheb} it holds true that
  \begin{align*}
    \P ( A_m ) &\le \| X_m \|^p_{L^p(\Omega;\R^d)} h_m^{1 - p \rho}    
  \end{align*}
  for all $m \in \N$. Consequently, 
  \begin{align*}
    \sum_{m = 1}^\infty \P(A_m) 
    \le \sum_{m = 1}^\infty \| X_m \|^p_{L^p(\Omega;\R^d)} h_m^{1 - p \rho} 
    \le \sup_{j \in \N} \Big(h_j^{-p \rho} \| X_j \|^p_{L^p(\Omega;\R^d)} \Big) 
    \sum_{m = 1}^\infty h_m < \infty
  \end{align*}
  due to our assumptions on $(X_m)_{m \in \N}$ and $(h_m)_{m \in \N}$. Thus, the
  Borel-Cantelli lemma (see Lemma~\ref{lem:BC}) yields
  $\P ( \limsup_{m \to \infty} A_m ) = 0$.
  Since
  \begin{align*}
    \limsup_{m \to \infty} A_m 
    &= \big\{ \omega \in \Omega \, : \, | X_m(\omega) | > h_m^{\rho -
    \frac{1}{p}} \text{ for infinitely many } m \in \N \big\}
  \end{align*}
  the assertion follows with $A \in \F$ being the complement of $\limsup_{m \to
  \infty} A_m \in \F$. Finally, the random variable $m_0 \colon \Omega
  \to \N_0$ is defined by
  \begin{align*}
    m_0(\omega) := \max\big( \{1\} \cup \big\{ m \in \N \, : \, |X_m(\omega)| >
    h_m^{\rho-\frac{1}{p}} \big\} \big) < \infty, \quad \text{ for all } \omega
    \in A,
  \end{align*}
  and $m_0(\omega) = 0$ for all $\omega \in \Omega \setminus A$.
\end{proof}

\begin{remark}
  The result of Lemma~\ref{lem3:pathwise} can equivalently be reformulated as
  follows: Under the same assumptions there exist a measurable set $A \in \F$
  with $\P(A) = 1$ and a nonnegative random variable $M \in L^p(\Omega;\R)$ such that
  for all $m \in \N$ and $\omega \in A$ we have
  \begin{align*}
    |X_m(\omega)| \le M(\omega) h_{m}^{\rho - \frac{1}{p}}.
  \end{align*}
  Hence, it is possible to relax the $\omega$-dependent step size restriction in
  Lemma~\ref{lem3:pathwise} in form of the random variable $m_0$ on the cost
  of introducing an $\omega$-dependent error constant $M(\omega)$.
  For further details we refer to the proof of \cite[Lemma~2.1]{kloeden2007}. 
\end{remark}

The proof of Theorem~\ref{th3:Rie_pathwise} is now a simple consequence of
Theorem~\ref{th4:randRiemann} and Lemma~\ref{lem3:pathwise}:

\begin{proof}[Proof of Theorem~\ref{th3:Rie_pathwise}]
  First, we assume that $g \in L^p([0,T];\R^d)$ for some $p \in (2,\infty)$.
  Let $(h_m)_{m \in \N}$ be an arbitrary sequence of step sizes with
  $\sum_{m = 1}^\infty h_m < \infty$. Then define 
  \begin{align*}
    X_m := \max_{n \in \{1,\ldots,N_{h_m}\}} \Big| \int_0^{t_n} g(s) \diff{s} -
    Q_{\tau,h_m}^n[g] \Big| 
  \end{align*}
  Clearly, $X_m \in L^p(\Omega;\R)$ for each $m \in \N$. In particular, from
  \eqref{eq4:errRie1} it follows that
  \begin{align*}
    \| X_m \|_{L^p(\Omega;\R)} \le 2 C_p T^{\frac{p-2}{2p}} \| g
    \|_{L^p(\Omega;\R)} h_m^{\frac{1}{2}}.    
  \end{align*}
  Thus, since $p > 2$ the conditions of Lemma~\ref{lem3:pathwise} are
  fulfilled with $\rho = \frac{1}{2}$ and assertion \eqref{eq4:errRie3}
  follows directly. 
  
  Next, if we additionally assume that $g \in \mathcal{C}^\gamma([0,T])$ for
  some $\gamma \in (0,1]$ then we immediately have $g \in L^p([0,T];\R^d)$ for
  every $p \in [2,\infty)$. Let $\epsilon \in (0,\frac{1}{2})$ be arbitrary.
  Choose a value for $p \in (2,\infty)$ such that $\frac{1}{p} < \epsilon$.
  Then, if we define $X_m$ as above we obtain from \eqref{eq4:errRie2} that
  \begin{align*}
    \| X_m \|_{L^p(\Omega;\R)} \le C_p \sqrt{T} \| g
    \|_{\mathcal{C}^\gamma([0,T])} h_m^{\frac{1}{2}+\gamma}
  \end{align*}
  for every $m \in \N$. Thus, a further application of
  Lemma~\ref{lem3:pathwise} with $\rho = \frac{1}{2}+\gamma$ yields
  \begin{align*}
    | X_m(\omega) | \le h_m^{\frac{1}{2}+\gamma- \frac{1}{p}} \le
    h_m^{\frac{1}{2}+\gamma- \epsilon} 
  \end{align*}
  for all $m \ge m_0^\epsilon(\omega)$ with probability one.
\end{proof}

\section{Numerical approximation of Carath\'eodory ODEs}
\label{sec:CDE}

In this section we investigate the numerical approximation of the exact
solution $u$ to the \emph{Carath\'eodory type} ordinary differential equation
\eqref{eq4:CDE}. In particular, we derive the order of convergence of the
randomized Euler method \eqref{eq5:NumMeth1} with respect to the norm in
$L^p(\Omega;\R^d)$. We also state the order of convergence in the almost sure
sense. Throughout this section, we shall allow the following assumptions on the
coefficient function $f$. 

\begin{assumption}
  \label{as4:f}
  The coefficient function $f \colon [0,T] \times \R^d \to
  \R^d$ is assumed to be measurable. Further, there exist $p \in [1,\infty]$
  and a measurable mapping $L \colon [0,T] \to [0,\infty)$ with 
  $\|L \|_{L^p([0,T];\R)} < \infty$ such that
  \begin{align}
    \label{eq4:Lip}
    | f(t,x_1) - f(t, x_2) | &\le L(t) | x_1 - x_2|
  \end{align}
  for almost all $t \in [0,T]$ and $x_1, x_2 \in \R^d$. In addition, there is a
  measurable mapping $K \colon [0,T] \to [0,\infty)$ with 
  $\|K \|_{L^p([0,T];\R)} < \infty$ such that
  \begin{align}
    \label{eq4:growth}
    | f(t,0) | \le K(t) 
  \end{align}
  for almost all $t \in [0,T]$.
\end{assumption}

Let us stress, that the mapping $f$ is not necessarily continuous with respect
to the temporal variable $t$. In addition, the mappings $L$ and $K$ are not
assumed to be bounded, in contrast to other results found in the literature
\cite{coulibaly1999,jentzen2009, stengle1990, stengle1995}. Moreover, from
\eqref{eq4:Lip} and \eqref{eq4:growth} we directly deduce the linear growth
condition 
\begin{align}
  \label{eq4:growth2}
  | f(t,x) | \le \overline{K}(t) ( 1 + |x| )
\end{align}
for almost all $t \in [0,T]$ and $x \in \R^d$. Here, $\overline{K} \colon [0,T] \to
[0,\infty)$ is the $L^p$-integrable mapping determined by $\overline{K}(t) :=
\max( K(t), L(t) )$, $t \in [0,T]$. Assumption~\ref{as4:f} is more than
sufficient to ensure the existence of a unique solution $u$ to the initial
value problem \eqref{eq4:CDE}, see \cite[Chap.~I, Thm~5.3]{hale1980}.

In the following proposition we collect a few properties of the solution $u$ to
\eqref{eq4:CDE}.

\begin{prop}
  \label{prop4:CDE}
  Let Assumption~\ref{as4:f} be fulfilled with $p \in [1,\infty]$. Then, the
  solution $u$ to the initial value problem \eqref{eq4:CDE} satisfies 
  \begin{align}
    \label{eq4:apriori}
    |u(t)| \le \Big( |u_0| + \int_0^T \overline{K}(s) \diff{s} \Big) \exp\Big(
    \int_0^t \overline{K}(s) \diff{s} \Big), \quad t \in [0,T].
  \end{align}
  Moreover, if $p \in (1,\infty]$ then for any $0 \le s \le t \le T$ it holds
  true that
  \begin{align}
    \label{eq4:cont}
    | u(t) - u(s) | \le \| \overline{K} \|_{L^p([0,T];\R)} \Big( 1 + \sup_{z
    \in [0,T]} |u(z)| \Big) |t - s|^{1 - \frac{1}{p}}.
  \end{align}
  In particular, $u$ is H\"older continuous with exponent $(1 - \frac{1}{p}) >
  0$.
\end{prop}

\begin{proof}
  Let $u$ be the solution to \eqref{eq4:CDE}. Then, from \eqref{eq4:sol}
  and \eqref{eq4:growth2} we get that
  \begin{align*}
    | u(t) | &\le |u_0| + \int_0^t | f(s, u(s) ) | \diff{s} \\
    &\le |u_0| + \int_0^t \overline{K}(s) ( 1 + |u(s)| ) \diff{s}\\
    &\le |u_0| + \int_0^T \overline{K}(s) \diff{s} + 
    \int_0^t \overline{K}(s) |u(s)| \diff{s}.
  \end{align*}
  Then, an application of Gronwall's inequality (see e.g. \cite[Chap.~I,
  Cor.~6.6]{hale1980}) yields the assertion \eqref{eq4:apriori}.

  Next, assume that $p \in (1,\infty)$ and let $0 \le s \le t \le T$ be
  arbitrary. Then, from \eqref{eq4:sol} and \eqref{eq4:growth2} we further
  deduce that 
  \begin{align*}
    | u(t) - u(s) | &\le \int_s^t |f(z, u(z)) | \diff{z} \le \int_s^t
    \overline{K}(z) ( 1 + |u(z)| ) \diff{z}. 
  \end{align*}
  Since $u$ is bounded and since the mapping  $\overline{K}$ is $p$-fold
  integrable we obtain from the H\"older inequality with exponents $p$ and $p'
  = \frac{p}{p-1} \in (1,\infty)$ that 
  \begin{align*}
    | u(t) - u(s) | &\le \int_0^T \one_{[s,t]}(z) \overline{K}(z) ( 1 + |u(z)| )
    \diff{z}\\
    &\le \Big( 1 + \sup_{z \in [0,T]} |u(z)| \Big) \Big( \int_0^T
    \overline{K}(z)^p \diff{z} \Big)^{\frac{1}{p}} |t - s|^{\frac{1}{p'}}
  \end{align*}
  Due to $\frac{1}{p'} = 1 - \frac{1}{p}$ this proves the asserted H\"older
  continuity of $u$ if $p \in (1,\infty)$. The case $p = \infty$ is treated
  similarly.
\end{proof}

Now we are well prepared to state the main result of this section.
The following theorem provides an error estimate of the randomized Euler method
\eqref{eq5:NumMeth1} under Assumption~\ref{as4:f} with respect to the norm in
$L^p(\Omega;\R^d)$. We give an explicit expression for the error
constant further below.

\begin{theorem}[$L^p$-error estimate]
  \label{th4:error1}
  Let Assumption~\ref{as4:f} be fulfilled with $p \in [2,\infty)$. Let $u$
  denote the exact solution to \eqref{eq4:CDE}. For given $h \in (0,1)$ let
  $(U^j)_{j \in \{0,1,\ldots,N_h\}}$ denote the numerical approximation
  determined by \eqref{eq5:NumMeth1} with initial condition $U^0 = u_0$.
  Then, there exists $C \in (0,\infty)$, independent of $h \in (0,1)$, such
  that 
  \begin{align*}
    \big\| \max_{n \in \{0,1,\ldots,N_h\}} | u(t_n) - U^n|
    \big\|_{L^p(\Omega;\R)} \le C h^{\frac{1}{2}}.
  \end{align*}
\end{theorem}

\begin{proof}
  Let $h \in (0,1)$ and $n \in \{1,\ldots, N_h\}$ be arbitrary. Since $U^0 =
  u_0$ and by using a telescopic sum argument as well as \eqref{eq4:sol} and
  \eqref{eq5:NumMeth1} and we get
  \begin{align*}
    u(t_n) - U^n 
    &= u(0) - U^0 + \sum_{j = 1}^{n} \big( u(t_j)-u(t_{j-1}) - (U^j-U^{j-1})
    \big)\\ 
    &= \sum_{j = 1}^n \Big( \int_{t_{j-1}}^{t_j} f(s, u(s)) \diff{s} - h
    f(t_{j-1} + \tau_j h, U^{j-1}) \Big).
  \end{align*}
  In order to simplify the notation we write $\theta_j := t_{j-1} + \tau_j h$.
  Note that the family of random variables $(\theta_j)_{j \in \N}$ is
  independent and $\theta_j$ is uniformly distributed on the interval
  $[t_{j-1}, t_j]$ for each $j \in \{1,\ldots,N_h\}$. Then, after adding and
  subtracting several terms we have to estimate the following three sums
  \begin{align}
    \label{eq4:err_rep}
    \begin{split}
      u(t_n) - U^n 
      &= \sum_{j = 1}^n \Big( \int_{t_{j-1}}^{t_j} f(s, u(s)) \diff{s} - h
      f(\theta_j, u(\theta_j) ) \Big) \\
      &\quad + h \sum_{j = 1}^n \Big( 
      f\big(\theta_j, u(\theta_j ) \big) - f \big( \theta_j, u( t_{j-1} ) \big) 
      \Big)\\
      &\quad + h \sum_{j = 1}^n \Big( 
      f\big(\theta_j, u( t_{j-1} )  \big) - f\big(\theta_j, U^{j-1}\big)    
      \Big)\\
      &=: S_1^n + S_2^n + S_3^n.
    \end{split}
  \end{align}
  First, we give an estimate of the term $S_3^n$. To this end we apply
  \eqref{eq4:Lip} and arrive at 
  \begin{align*}
    | S_3^n | &\le h \sum_{j = 1}^n \Big| 
    f\big(\theta_j, u( t_{j-1} )  \big) - f\big( \theta_j, U^{j-1} \big)
    \Big|\\
    &\le h \sum_{j = 1}^n L(\theta_j) \big| u( t_{j-1} ) - U^{j-1} \big|\\
    &\le h \sum_{j = 1}^n L(\theta_j) \max_{i \in \{0,1,\ldots,j-1\}}
    \big| u( t_{i} ) - U^{i} \big|.
  \end{align*}
  Observe that this inequality is only valid in the almost sure sense, since
  \eqref{eq4:Lip} holds for almost all $t \in [0,T]$. However, this is
  sufficient, since the expected value will eventually be applied.
  Therefore, after taking the Euclidean norm $| \cdot |$ and the maximum over
  the time levels in \eqref{eq4:err_rep} we obtain
  \begin{align*}
    \max_{i \in \{0,1,\ldots,n\}} \big| u(t_i) - U^i \big|
    &\le \max_{j \in \{1,\ldots,N_h\}} | S_1^j|
    + \max_{j \in \{1,\ldots,N_h\}} | S_2^j|\\
    &\quad + h \sum_{j = 1}^n L(\theta_j) \max_{i \in \{0,1,\ldots,j-1\}} 
    \big| u( t_{i} ) - U^{i} \big|
  \end{align*}
  almost surely for every $n \in \{1,\ldots,N_h\}$. 
  Next, we apply the $p$-th power of the
  $L^p(\Omega;\R)$-norm to both sides of the inequality. From the
  fact that $(a + b)^{p} \le 2^{p-1} ( a^p + b^p )$ for all $a,b \in
  [0,\infty)$ we then get 
  \begin{align}
    \label{eq4:gronwall}
    \begin{split}
      &\big\| \max_{i \in \{0,1,\ldots,n\}} \big| u(t_i) - U^i \big|
      \big\|_{L^p(\Omega;\R)}^p \\
      &\quad \le  2^{p - 1} \big\| \max_{n \in \{1,\ldots,N_h\}} |S_1^n|
      + \max_{n \in \{1,\ldots,N_h\}} |S_2^n| \big\|_{L^p(\Omega;\R)}^p\\
      &\qquad + 2^{p-1} \Big\|h \sum_{j = 1}^n L(\theta_j) \max_{i \in
      \{0,1,\ldots,j-1\}} \big| u( t_{i} ) - U^{i} \big|
      \Big\|_{L^p(\Omega;\R)}^p. 
    \end{split}
  \end{align}
  The last term is further estimated by H\"older's inequality as follows
  \begin{align*}
    &\Big\|h \sum_{j = 1}^n 1 \cdot L(\theta_j) \max_{i \in
    \{0,1,\ldots,j-1\}} \big| u( t_{i} ) - U^{i} \big|
    \Big\|_{L^p(\Omega;\R)}^p\\ 
    &\quad \le h^p n^{p-1} \sum_{j = 1}^n \big\| L(\theta_j) \max_{i \in
    \{0,1,\ldots,j-1\}} \big| u( t_{i} ) - U^{i} \big|
    \big\|_{L^p(\Omega;\R)}^p.
  \end{align*}
  For the next step, first take note of
  \begin{align}
    \label{eq4:Lp}
    \E \big[ L(\theta_j)^p \big] = \frac{1}{h} \int_{t_{j-1}}^{t_{j}} L(s)^p
    \diff{s},
  \end{align} 
  since $\theta_j \sim \mathcal{U}([t_{j-1},t_j])$. Moreover, we observe that
  $\theta_j$, and therefore also $L(\theta_j)$, is independent of the errors at
  earlier time levels. Thus, from \eqref{eq:prod_ind} we obtain 
  \begin{align*}
    &\big\| L(\theta_j) \max_{i \in \{0,1,\ldots,j-1\}}
    \big| u( t_{i} ) - U^{i} \big|  \big\|_{L^p(\Omega;\R)}^p\\
    &\quad = \E \big[ L(\theta_j)^p \big] 
    \E \big[ \max_{i \in \{0,1,\ldots,j-1\}}
    \big| u( t_{i} ) - U^{i} \big|^p \big]\\
    &\quad = \frac{1}{h} \int_{t_{j-1}}^{t_j} L(s)^p \diff{s} 
    \big\|  \max_{i \in \{0,1,\ldots,j-1\}} \big| u( t_{i} ) - U^{i} \big|
    \big\|_{L^p(\Omega;\R)}^p.
  \end{align*}  
  Inserting this into \eqref{eq4:gronwall} yields
  \begin{align*}
     &\big\| \max_{i \in \{0,1,\ldots,n\}} \big| u(t_i) - U^i \big|
     \big\|_{L^p(\Omega;\R)}^p \\
     &\quad \le  2^{p - 1} \big\| \max_{n \in \{1,\ldots,N_h\}} |S_1^n|
     + \max_{n \in \{1,\ldots,N_h\}} |S_2^n| \big\|_{L^p(\Omega;\R)}^p\\
     &\qquad + 2^{p-1} T^{p-1} \sum_{j = 1}^n
     \int_{t_{j-1}}^{t_j} L(s)^p \diff{s} 
     \big\|  \max_{i \in \{0,1,\ldots,j-1\}} \big| u( t_{i} ) - U^{i} \big|
     \big\|_{L^p(\Omega;\R)}^p.
  \end{align*}
  Therefore, an application of Lemma~\ref{lem:Gronwall} 
  results in
  \begin{align*}
    &\big\| \max_{i \in \{0,1,\ldots,n\}} \big| u(t_i) - U^i \big|
    \big\|_{L^p(\Omega;\R)}^p\\
    &\quad \le 2^{p - 1} \big\| \max_{n \in \{1,\ldots,N_h\}} |S_1^n|
     + \max_{n \in \{1,\ldots,N_h\}} |S_2^n| \big\|_{L^p(\Omega;\R)}^p
    \exp\Big( (2T)^{p-1} \| L \|_{L^p([0,T];\R)}^p \Big).
  \end{align*}
  It remains to give estimates for the terms $S_1^n$ and $S_2^n$ with
  respect to the $L^p(\Omega;\R^d)$-norm. For this we observe that the sum
  $S_1^n$ is the error of a randomized Riemann sum approximation. Since by
  \eqref{eq4:growth2}  
  \begin{align*}
    \| f( \cdot, u(\cdot) ) \|_{L^p([0,T];\R^d)}
    &\le \Big( \int_0^T \overline{K}^p(s) ( 1 + |u(s)| )^p \diff{s}
    \Big)^{\frac{1}{p}}\\
    &\le \Big( 1 + \sup_{t \in [0,T]} |u(t)| \Big) \big\|
    \overline{K} \big\|_{L^p([0,T];\R)} < \infty,
  \end{align*}
  Theorem~\ref{th4:randRiemann} is applicable and we deduce from
  \eqref{eq4:errRie1} that
  \begin{align*}
    \big\| \max_{n \in \{1,\ldots,N_h\}} | S_1^n|
    \big\|_{L^p(\Omega;\R^d)} \le 2 C_p T^{\frac{p-2}{2p}} 
    \Big( 1 + \sup_{t \in [0,T]} |u(t)| \Big) \big\|
    \overline{K} \big\|_{L^p([0,T];\R)} h^{\frac{1}{2}}.
  \end{align*}
  Regarding the estimate of $S_2^n$ we make use of \eqref{eq4:Lip} and the
  $(1 - \frac{1}{p})$-H\"older continuity of $u$ from \eqref{eq4:cont}. Then we
  obtain 
  \begin{align}
    \label{eq4:estS2}
    \begin{split}
      \max_{n \in \{1,\ldots,N_h\}} |S_2^n| 
      &\le h \sum_{j = 1}^{N_h} \Big| f\big(\theta_j, u(\theta_j ) \big) 
      - f\big(\theta_j, u( t_{j-1} )  \big) \Big| \\
      &\le h \sum_{j = 1}^{N_h} L(\theta_j) \big| u(\theta_j ) - u( t_{j-1} )
      \big| \\
      &\le \big\| \overline{K} \big\|_{L^p([0,T];\R)} \Big( 1 + \sup_{z \in
      [0,T]} | u(z) | \Big) h^{2 - \frac{1}{p}} \sum_{j = 1}^{N_h} L(\theta_j),
    \end{split}
  \end{align}
  where, as already noted above, this inequality is only valid in the almost
  sure sense. Next, by an application of H\"older's inequality it holds true that
  \begin{align*}
    h \sum_{j = 1}^{N_h} L(\theta_j) \le T^{1 - \frac{1}{p}} \Big( h \sum_{j =
    1}^{N_h} L(\theta_j)^p \Big)^{\frac{1}{p}}.
  \end{align*}
  Together with \eqref{eq4:Lp} we conclude from \eqref{eq4:estS2} that
  \begin{align*}
    &\big\| \max_{n \in \{1,\ldots,N_h\}} |S_2^n| \big\|_{L^p(\Omega;\R)}\\
    &\quad \le T^{1 - \frac{1}{p}} \big\| \overline{K} \big\|_{L^p([0,T];\R)}
    \| L \|_{L^p([0,T];\R)} \Big( 1 + \sup_{z \in [0,T]} | u(z) | \Big)  
    h^{1 - \frac{1}{p}}.
  \end{align*}
  This completes the proof.
\end{proof}

\begin{remark}
  Let us mention, that the proof of Theorem~\ref{th4:error1} admits 
  an explicit expression of the error constant $C$, namely
  \begin{align*}
    C &= 2^{1 - \frac{1}{p}} T^{\frac{1}{2} - \frac{1}{p}} \big\| \overline{K}
    \big\|_{L^p([0,T];\R)} 
    \exp\Big( \frac{1}{p} (2T)^{p-1} \| L \|_{L^p([0,T];\R)}^p \Big)\\
    &\quad \times \Big( 2 C_p 
    + T^{\frac{1}{2}} \| L \|_{L^p([0,T];\R)} \Big) \Big( 1 + \sup_{z \in [0,T]}
    | u(z) | \Big).  
  \end{align*}
  One could further estimate the supremum of $u$ by \eqref{eq4:apriori}.
  
  We observe that the error constant $C$ grows at least exponentially with $T$
  and $\|L\|_{L^p([0,T];\R)}$. This indicates that the numerical method
  requires very small values for the step size $h$ if applied to initial value
  problems on large time intervals $T \gg 1$ or with huge Lipschitz bounds
  $\|L\|_{L^p([0,T];\R)} \gg 1$.
\end{remark}

In the same way as in Theorem~\ref{th3:Rie_pathwise} we also have a result on
the almost sure convergence of the randomized Euler method \eqref{eq5:NumMeth1}.
Compare also with \cite[Theorem~2]{jentzen2009}, if the coefficient function
$f$ is additionally assumed to be locally bounded.

\begin{theorem}[Almost sure convergence]
  \label{th4:error2}
  Let Assumption~\ref{as4:f} be fulfilled with $p \in (2,\infty)$ and let $u$
  denote the exact solution to \eqref{eq4:CDE}. For a given sequence 
  $(h_m)_{m \in \N} \subset (0,1)$ of step sizes with $\sum_{m = 1}^\infty h_m
  < \infty$ let $(U^j_{m})_{j \in \{0,1,\ldots,N_{h_m}\}}$ denote the
  numerical approximation determined by \eqref{eq5:NumMeth1} with initial
  condition $U^0_{m} = u_0$ and step size $h_m$, $m \in \N$.
  Then, there exist a random variable $m_0 \colon \Omega \to \N_0$ and a
  measurable set $A \in \F$ with $\P(A) = 1$ such that for every $\omega \in A$
  and $m \ge m_0(\omega)$ it holds true that
  \begin{align*}
    \max_{n \in \{0,1,\ldots,N_{h_m}\}} | u(t_n) - U^n_m(\omega)|
    \le h_m^{\frac{1}{2} - \frac{1}{p}}.
  \end{align*}
\end{theorem}

Since the proof of Theorem~\ref{th4:error2} follows from the same steps as the 
proof of the first part of Theorem~\ref{th3:Rie_pathwise}, it is omitted.

\section{Randomized Runge-Kutta methods for ODEs}
\label{sec:ODE}

In this section, we consider initial value problems \eqref{eq4:CDE} whose
coefficient function $f$ enjoys slightly more regularity with respect to the
temporal variable $t$ than those considered in Section~\ref{sec:CDE}. However,
we still do not assume any differentiability of $f$.

\begin{assumption}
  \label{as5:f}
  The coefficient function $f \colon [0,T] \times \R^d \to
  \R^d$ is assumed to be continuous. Further, there exists $L \in (0,\infty)$
  such that  
  \begin{align}
    \label{eq5:Lip}
    | f(t,x_1) - f(t, x_2) | &\le L | x_1 - x_2|
  \end{align}
  for all $t \in [0,T]$ and $x_1, x_2 \in \R^d$. In addition, there
  exist $K \in (0,\infty)$ and $\gamma \in (0,1]$ with
  \begin{align}
    \label{eq5:Hoelder}
    | f(t_1, x) - f(t_2, x) | &\le K \big( 1 + |x| \big) | t_1 - t_2|^\gamma 
  \end{align}
  for all $t_1,t_2 \in [0,T]$ and $x \in \R^d$.
\end{assumption}

As a direct consequence of Assumption~\ref{as5:f} we take note of the linear
growth bound
\begin{align*}
  |f(t,x)| \le \overline{K} \big(1 + |x| \big), \quad \text{ for all } t \in
  [0,T], \, x \in \R^d,   
\end{align*}
with $\overline{K} := \max( L, KT^{\gamma} + |f(0,0)|)$.
Clearly, under Assumption~\ref{as5:f} the initial value problem \eqref{eq4:CDE}
is a classical ordinary differential equation. Therefore, there exists a
(global) unique solution $u \colon [0,T] \to \R^d$. In particular, the solution
$u$ is continuously differentiable with
\begin{align}
  \label{eq5:apriori}
  | u(t) |
  &\le \big( |u_0| + \overline{K} T \big) \exp\big( \overline{K} t \big)  
\end{align}
for all $t \in [0,T]$, and
\begin{align}
  \label{eq5:cont}
  | u(t) - u(s) | &\le \overline{K} \Big( 1 + \sup_{z \in [0,T]} |u(z)| \Big)
  |t - s|
\end{align}
for all $t,s \in [0,T]$.

The following theorem contains the error estimates for the randomized Euler
method \eqref{eq5:NumMeth1} and the randomized Runge-Kutta method
\eqref{eq5:NumMeth2} under Assumption~\ref{as5:f}. We provide explicit
expressions for the error constants further below.

\begin{theorem}[$L^p$-error estimate]
  \label{th5:error1}
  Let Assumption~\ref{as5:f} be fulfilled with $\gamma \in (0,1]$. Let $u$
  be the exact solution to \eqref{eq4:CDE}. For given step size $h \in (0,1)$
  we denote by $(U^j)_{j \in \{0,1,\ldots,N_h\}}$ and $(V^j)_{j
  \in \{0,1,\ldots,N_h\}}$ the sequences generated by the numerical methods 
  \eqref{eq5:NumMeth1} and \eqref{eq5:NumMeth2}, respectively. Then, for every
  $p \in [2,\infty)$ there exists a constant $C_U \in (0,\infty)$, independent
  of $h \in (0,1)$, such that 
  \begin{align}
    \label{eq5:error1}
    \big\| \max_{n \in \{0,1,\ldots,N_h\}} | u(t_n) - U^n|
    \big\|_{L^p(\Omega;\R)} \le C_U h^{\min(\frac{1}{2}+\gamma, 1)}.
  \end{align}
  Moreover, for every $p \in [2,\infty)$ there exists a constant $C_V \in
  (0,\infty)$, independent of $h \in (0,1)$, such that 
  \begin{align}
    \label{eq5:error2}
    \big\| \max_{n \in \{0,1,\ldots,N_h\}} | u(t_n) - V^n|
    \big\|_{L^p(\Omega;\R)} \le C_V h^{\frac{1}{2}+\gamma}.
  \end{align}
\end{theorem}

\begin{proof}
  Let $h \in (0,1)$ be an arbitrary step size. As in the proof of
  Theorem~\ref{th4:error1} we write $\theta_j := t_{j-1} + \tau_j h$ for every
  $j \in \{1,\ldots,N_h\}$. 
  
  We first prove the error estimate \eqref{eq5:error1}
  for the randomized Euler method \eqref{eq5:NumMeth1}. 
  For this let $n \in \{1,\ldots, N_h\}$ be arbitrary.
  As in the proof of Theorem~\ref{th4:error1} in \eqref{eq4:err_rep} we 
  split the error into three sums of the form
  \begin{align}
    \label{eq5:err_rep}
    \begin{split}
      u(t_n) - U^n 
      &= \sum_{j = 1}^n \Big( \int_{t_{j-1}}^{t_j} f(s, u(s)) \diff{s} - h
      f(\theta_j, u(\theta_j) ) \Big) \\
      &\quad + h \sum_{j = 1}^n \Big( 
      f\big(\theta_j, u(\theta_j ) \big) 
      - f\big(\theta_j, u( t_{j-1} ) \big) \Big)\\
      &\quad + h \sum_{j = 1}^n \Big( 
      f\big(\theta_j, u( t_{j-1} ) \big) - f(\theta_j, U^{j-1})    
      \Big)\\
      &=: S_1^n + S_2^n + S_3^n.
    \end{split}
  \end{align}
  Due to \eqref{eq5:Lip} we can estimate the term $S_3^n$ by
  \begin{align*}
    |S_3^n| &\le h \sum_{j = 1}^n \Big| f\big(\theta_j, u( t_{j-1} ) \big) -
    f(\theta_j, U^{j-1}) \Big|\\
    &\le L h \sum_{j = 1}^n \max_{i \in \{0,1,\ldots, j-1\}} \big| u( t_{i} )
    - U^{i} \big|. 
  \end{align*}
  Thus, applying the Euclidean norm and then taking the maximum over all
  time steps $n$ in \eqref{eq5:err_rep} yields
  \begin{align*}
    \max_{i \in \{0,1,\ldots,n\}} \big| u(t_i) - U^i \big|
    &\le \max_{j \in \{1,\ldots,N_h\}} |S_1^j| 
    + \max_{j \in \{1,\ldots,N_h\}} |S_2^j| \\
    &\quad + L h \sum_{j = 1}^{n} \max_{i \in \{0,1,\ldots, j-1\}} \big| u(
    t_{i} ) - U^{i} \big|.  
  \end{align*}
  In contrast to the situation in Theorem~\ref{th4:error1} the Lipschitz
  constant $L$ is now deterministic. Thus, after applying the
  $L^p(\Omega;\R)$-norm to this inequality we obtain
  \begin{align*}
    \big\| \max_{i \in \{0,1,\ldots,n\}} \big| u(t_i) - U^i \big|
    \big\|_{L^p(\Omega;\R)}
    &\le \big\| \max_{j \in \{1,\ldots,N_h\}} |S_1^j| 
    + \max_{j \in \{1,\ldots,N_h\}} |S_2^j| \big\|_{L^p(\Omega;\R)}\\
    &\quad + L h \sum_{j = 1}^{n} \big\| \max_{i \in \{0,1,\ldots, j-1\}} \big|
    u( t_{i} ) - U^{i} \big| \big\|_{L^p(\Omega;\R)}.
  \end{align*}
  Then, an application of Gronwall's lemma (see Lemma~\ref{lem:Gronwall})
  yields
  \begin{align*}
    &\big\| \max_{i \in \{0,1,\ldots,N_h\}} \big| u(t_i) - U^i \big|
    \big\|_{L^p(\Omega;\R)} \\
    &\quad \le \big\| \max_{j \in \{1,\ldots,N_h\}} |S_1^j| 
    + \max_{j \in \{1,\ldots,N_h\}} |S_2^j| \big\|_{L^p(\Omega;\R)}
    \exp\big( L T \big)
  \end{align*}
  and it remains to estimate the norms of the sums $S_1^n$ and $S_2^n$.

  Regarding the term $S_1^n$ it follows from \eqref{eq5:Lip}
  and \eqref{eq5:Hoelder} that
  \begin{align}
    \label{eq5:Hoelder_f}
    \begin{split}
      \big| f(s,u(s)) - f(t,u(t)) \big| &\le \big| f(s,u(s)) - f(s,u(t)) \big|
      + \big| f(s,u(t)) - f(t,u(t)) \big| \\
      &\le L \big| u(s) - u(t) \big| + K \Big( 1 + \sup_{z \in [0,T]} |u(z)|
      \Big) |t - s|^\gamma  
    \end{split}
  \end{align}
  for all $t,s \in [0,T]$. Hence, due to \eqref{eq5:cont}
  we see that the mapping $[0,T] \ni t \mapsto f(t,u(t)) \in \R^d$ is 
  $\gamma$-H\"older continuous. In particular, 
  \begin{align*}
    \big\| f(\cdot, u(\cdot)) \big\|_{\mathcal{C}^\gamma([0,T])} 
    \le \big(2 + L T^{1 - \gamma}\big)
    \overline{K} \Big( 1 + \sup_{z \in [0,T]} |u(z)| \Big). 
  \end{align*}
  Therefore, we can apply the estimate \eqref{eq4:errRie2} from
  Theorem~\ref{th4:randRiemann} to $S_1^n$. This gives
  \begin{align*}
    \big\| \max_{j \in \{1,\ldots,N_h\}} |S_1^j| \big\|_{L^p(\Omega;\R)}
    \le C_p \sqrt{T} \big\| f(\cdot, u(\cdot))
    \big\|_{\mathcal{C}^\gamma([0,T])} h^{\frac{1}{2} + \gamma}.
  \end{align*}
  Finally, the estimate of $S_2^n$ follows the same lines as in
  \eqref{eq4:estS2} but we additionally make use of the Lipschitz continuity
  \eqref{eq5:cont} of $u$. Then we get
  \begin{align*}
    \max_{j \in \{1,\ldots,N_h\}} |S_2^j| 
    &\le L h \sum_{j = 1}^{N_h} \big| u(\theta_j ) -  u( t_{j-1} ) \big|
    \le L \overline{K} T \Big( 1 + \sup_{z \in [0,T]} |u(z)| \Big) h.
  \end{align*}
  This completes the proof of \eqref{eq5:error1}.

  Let us now turn to the proof of the error estimate \eqref{eq5:error2} for the
  randomized Runge-Kutta method \eqref{eq5:NumMeth2}. This time
  we apply a slightly modified version of
  \eqref{eq5:err_rep}: 
  \begin{align}
    \label{eq5:err_rep2}
    \begin{split}
      u(t_n) - V^n 
      &= \sum_{j = 1}^n \Big( \int_{t_{j-1}}^{t_j} f(s, u(s)) \diff{s} - h
      f(\theta_j, u(\theta_j) ) \Big) \\
      &\quad + h \sum_{j = 1}^n \Big( 
      f\big(\theta_j, u(\theta_j ) \big) 
      - f\big(\theta_j, u( t_{j-1} ) + h\tau_j f(t_{j-1},u(t_{j-1})) \big)
      \Big)\\ 
      &\quad + h \sum_{j = 1}^n \Big( 
      f\big(\theta_j, u( t_{j-1} ) + h\tau_j f(t_{j-1},u(t_{j-1})) \big) 
      - f(\theta_j, V^{j-1}_\tau) \Big)\\
      &=: S_4^n + S_5^n + S_6^n.
    \end{split}
  \end{align}
  Note that actually $S_4^n = S_1^n$ for all $n \in \{1,\ldots,N_h\}$. Thus we
  directly obtain
  \begin{align*}
    \big\| \max_{j \in \{1,\ldots,N_h\}} |S_4^j| \big\|_{L^p(\Omega;\R)}
    \le C_p \sqrt{T} \big\| f(\cdot, u(\cdot))
    \big\|_{\mathcal{C}^\gamma([0,T])} h^{\frac{1}{2} + \gamma}.    
  \end{align*}
  Moreover, due to \eqref{eq5:Hoelder_f} the estimate of $S_5^n$ reads as
  follows 
  \begin{align*}
    \max_{n \in \{1,\ldots,N_h\}} |S_5^n| 
    &\le L h \sum_{j = 1}^{N_h} \big| u(\theta_j ) 
    -  u( t_{j-1} ) - h\tau_j f(t_{j-1},u(t_{j-1})) \big|\\
    &\le L h \sum_{j = 1}^{N_h} \int_{t_{j-1}}^{t_{j-1} + h\tau_j} 
    \big| f(s,u(s)) - f(t_{j-1},u(t_{j-1})) \big| \diff{s} \\
    &\le L T \big\| f(\cdot, u(\cdot)) \big\|_{\mathcal{C}^\gamma([0,T])}
    h^{1 + \gamma}.
  \end{align*}
  For the last step recall the definition of $V_\tau^n$ from
  \eqref{eq5:NumMeth2}. Thus, by using \eqref{eq5:Lip} we get
  \begin{align*}
    &\Big| f\big(\theta_j, u( t_{j-1} ) + h\tau_j f(t_{j-1},u(t_{j-1})) \big) 
      - f(\theta_j, V^{j-1}_\tau) \Big|\\
    &\quad \le L \big| u( t_{j-1} ) + h\tau_j f(t_{j-1},u(t_{j-1})) - 
    \big( V^{j-1} + h\tau_j f(t_{j-1}, V^{j-1} ) \big) \big|\\
    &\quad \le L (1 + h L) \big| u( t_{j-1} ) - V^{j-1} \big|.
  \end{align*}
  Consequently, since $h \in (0,1)$ we have
  \begin{align*}
    | S_6^n | \le L ( 1 + L) h  \sum_{j = 1}^n \max_{i \in \{1,\ldots,j-1\}}
    \big|  u( t_{i} ) - V^{i} \big|
  \end{align*}
  for every $n \in \{1,\ldots,N_h\}$. Then, 
  the error estimate \eqref{eq5:error2} follows from a further application 
  of Lemma~\ref{lem:Gronwall} as demonstrated above.
\end{proof}

\begin{remark}
  As in the previous section, the 
  proof of Theorem~\ref{th5:error1} also admits an explicit expressions of
  the error constants $C_U$ and $C_V$, namely
  \begin{align*}
    C_U &= \exp(LT)  \overline{K} \sqrt{T}
    \Big( C_p \big(2 + L T^{1 - \gamma}\big) +  L \sqrt{T} \Big)
    \Big( 1 + \sup_{z \in [0,T]} |u(z)| \Big)
  \end{align*}
  and
  \begin{align*}
    C_V &= \exp\big( L (1+L) T \big) \overline{K} \big( C_p \sqrt{T} + L T \big)
    \big(2 + L T^{1 - \gamma}\big) \Big( 1 + \sup_{z \in [0,T]}
    |u(z)| \Big),
  \end{align*}
  where the supremum of $u$ can be further estimated by \eqref{eq5:apriori}.
    
  Again, we take note of the fact that the error constants $C_U$ and $C_V$ both
  grow at least exponentially with the final time $T$ and the Lipschitz
  constant $L$. Both methods are therefore not necessarily well-suited for
  long-time simulations or if the ODE is stiff.
\end{remark}

We close this section with the following result on the almost sure convergence
of the randomized Euler method \eqref{eq5:NumMeth1} and the 
randomized Runge-Kutta method \eqref{eq5:NumMeth2}.

\begin{theorem}[Almost sure convergence]
  \label{th5:error2}
  Let Assumption~\ref{as5:f} be fulfilled for some $\gamma \in (0,1]$ and let
  $u$ denote the exact solution to \eqref{eq4:CDE}. For a given sequence 
  $(h_m)_{m \in \N} \subset (0,1)$ of step sizes with $\sum_{m = 1}^\infty h_m
  < \infty$ let $(U^j_{m})_{j \in \{0,1,\ldots,N_{h_m}\}}$ 
  and $(V^j_{m})_{j \in \{0,1,\ldots,N_{h_m}\}}$
  denote the numerical approximations determined by \eqref{eq5:NumMeth1} and
  \eqref{eq5:NumMeth2} with initial condition $u_0$ and step size
  $h_m$, $m \in \N$, respectively.
  Then, for every $\epsilon \in (0,\frac{1}{2})$ there exist a random variable
  $m_U \colon \Omega \to \N_0$ and a measurable set $A_U \in \F$ with
  $\P(A_U) = 1$ such that for every $\omega \in A_U$
  and $m \ge m_U(\omega)$ we have
  \begin{align*}
    \max_{n \in \{0,1,\ldots,N_{h_m}\}} | u(t_n) - U^n_m(\omega)|
    \le h_m^{\min(1, \frac{1}{2} + \gamma) - \epsilon}.
  \end{align*}
  In addition, for every $\epsilon \in (0,\frac{1}{2})$ there exist a random
  variable $m_V \colon \Omega \to \N_0$ and a measurable set $A_V \in \F$ with
  $\P(A_V) = 1$ such that for every $\omega \in A_V$ and $m \ge m_U(\omega)$ we
  have
  \begin{align*}
    \max_{n \in \{0,1,\ldots,N_{h_m}\}} | u(t_n) - V^n_m(\omega)|
    \le h_m^{\frac{1}{2} + \gamma - \epsilon}.
  \end{align*}
\end{theorem}

The proof of Theorem~\ref{th5:error2} is similar to the proof of the
second part of Theorem~\ref{th3:Rie_pathwise} and is therefore omitted.


\section{Numerical Examples}
\label{sec:examples}

In this section we illustrate our theoretical results through a few numerical
experiments.

\subsection{State-independent case with weak singularities}
\label{sec:example1}
Consider the following ODE with a state-independent coefficient function 
\begin{align}
  \label{eq:ODEexample1}
  \begin{split}
    \begin{cases}
      \dot{u}(t) &= (T-t)^{-1/\gamma}, \quad t \in [0,T],\\
      u(0) &= 0,
    \end{cases}
  \end{split}
\end{align}
with varying values for the parameter $\gamma$. In dependence of $\gamma$
we have different regularity of the coefficient function
$g(t):=(T-t)^{-1/\gamma}$ in terms of the $L^p$-spaces. It is not hard to see
that the exact solution at the final time $T$ is given by $
\frac{T}{1-1/\gamma}$. In the experiment, we take $\gamma$ to be $2,3,5,8$ and
$10$, $T=1$ and simulate the solutions via 
scheme \eqref{eq5:NumMeth1}, which in fact simplifies to the randomized Riemann
sum \eqref{eq3:RandRie}. We approximate the error of the quadrature rule with
respect to the $L^2$-norm at terminal time $T=1$ by a Monte Carlo simulation with $1000$
independent samples. The result is shown in Figure \ref{fig:1}.

\begin{figure}[h]
\centering
\caption{{\small $L^2$ convergence of the randomized Riemann sum to Eqn.
\eqref{eq:ODEexample1}}} 
\includegraphics[width=0.9\textwidth]{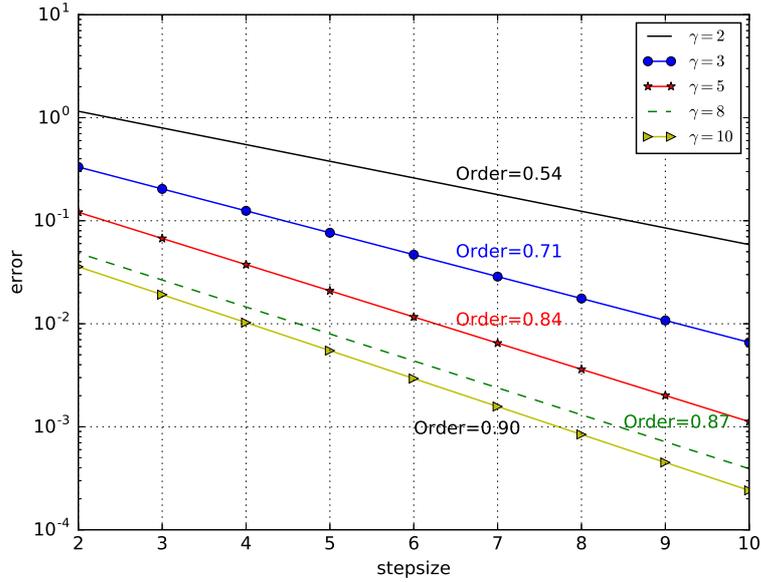}
\label{fig:1}
\end{figure}

According to Theorem \ref{th4:randRiemann}, the convergence order depends on
the integrability of the function $g$. In Figure \ref{fig:1}, the 
root-mean-squared errors were plotted versus the 2-logarithm of the underlying
step size, i.e., the number $n$ on the x-axis indicates the step size $h =
2^{-n}$. When $\gamma=0.5$, the observed order of convergence is as expected around
$\frac{1}{2}$, since $g$ is only $L^{2 - \epsilon}$ integrable. When increasing
the value for $\gamma$ from $2$ to $10$, the regularity of $g$ is 
raised, which in turn gives an increase in the observed order of convergence
from $0.54$ to $0.90$.  

\subsection{$L^2$ convergence for an ODE with jumps}
\label{sec:example3}
Consider the following ODE with a non-continuous coefficient function:
 \begin{align}
  \label{eq:ODEexample3}
  \begin{split}
    \begin{cases}
      \dot{u}(t) &= g(t)u, \quad t \in [0,T],\\
      u(0) &= 1,
    \end{cases}
  \end{split}
\end{align}
where  $g(t):=\left[- \frac{1}{10} \mbox{sgn}(\frac{1}{4}
T-t)-\frac{1}{5} \mbox{sgn}(\frac{1}{2} T-t)- \frac{7}{10}
\mbox{sgn}(\frac{3}{4}T-t)\right]$\ and 
 \begin{align}
  \mbox{sgn}(t):=
  \begin{split}
    \begin{cases}
      -1,& \ \mbox{if\ }t<0,\\
       0,& \ \mbox{if\ }t=0,\\      
       1,& \ \mbox{if\ }t>0.
    \end{cases}
  \end{split}
\end{align}
Here we have three jump points at $t=\frac{1}{4} T$, $t=\frac{1}{2} T$ and
$t=\frac{3}{4}T$. It is easy to see that the exact solution at terminal time
equals $\exp(-\frac{3}{10}T)$. We perform the numerical experiment  
with the classical Euler scheme, the randomized Euler scheme
\eqref{eq5:NumMeth1} and the randomized Runge-Kutta scheme
\eqref{eq5:NumMeth2}, respectively. A comparison of the $L^2$-errors at the
final time $T = 1$ is shown in Figure \ref{fig:3}, where the errors have been
approximated by a Monte Carlo simulation with $1000$ independent samples for
the same step sizes as in Section~\ref{sec:example1}.   

\begin{figure}[h]
  \centering
  \caption{{\small $L^2$-errors versus step sizes for Eqn. (\ref{eq:ODEexample3})}}
  \includegraphics[width=0.9\textwidth]{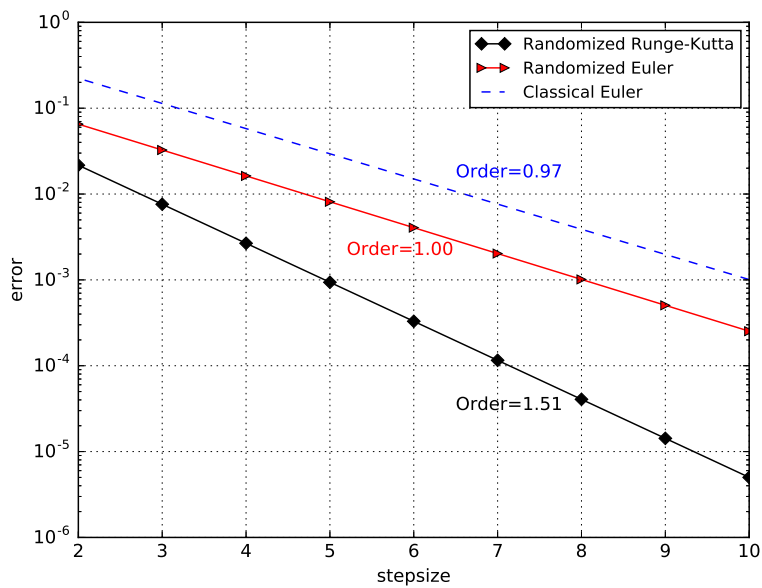}
  \label{fig:3}
\end{figure}

Note that by our choice of the step sizes the classical Euler scheme always
evaluates the mapping $g$ at the three jump points. But due to the definition
of the sign function, one of the summands in the definition of $g$ is always
equal to zero at the jump points, causing $g$ to be neither left continuous nor
right continuous at these points. For instance, we have $g(\frac{1}{2}) =
-\frac{3}{5}$, while $g(\frac{1}{2}+\epsilon) = -\frac{2}{5}$ and
$g(\frac{1}{2}-\epsilon) = -\frac{4}{5}$ for all $\epsilon \in
(0,\frac{1}{4})$. This causes an additional error of order $h$ in each step of
the classical Euler scheme, where a jump point of $g$ is involved.

On the other hand, this type of error is avoided by both randomized numerical
methods, since the random variable $\tau$ will prevent the evaluation of $g$ at
jump points almost surely. This explains why both randomized methods perform
better than the classical method if we compare the $L^2$-errors for the same
step sizes. Further, although the coefficient 
function $g$ is not continuous with respect to the time variable, we
observe an experimental convergence of order $1.51$ for the randomized
Runge-Kutta method \eqref{eq5:NumMeth2}. This is well in agreement with the
maximum order of convergence that has been proven for that method in
Theorem~\ref{th5:error1}. 

\begin{figure}[h]
  \centering
  \caption{CPU time versus $L^2$ errors}
  \includegraphics[width=0.9\textwidth]{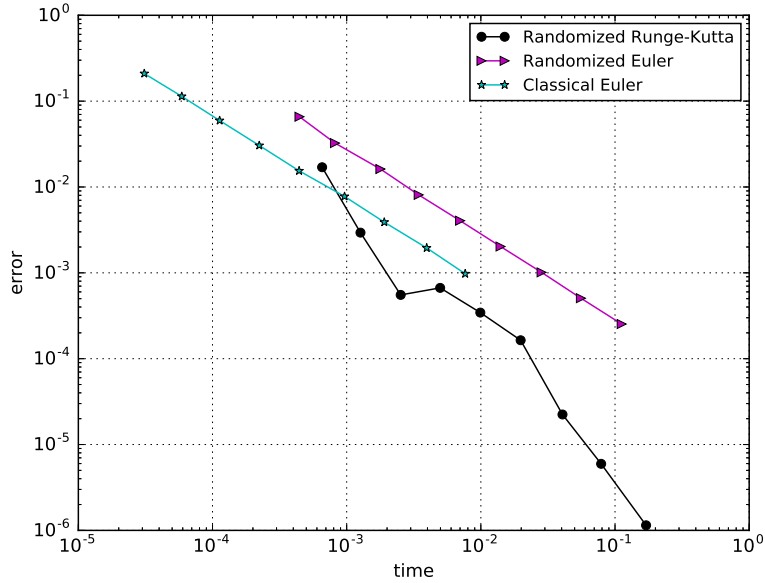}
  \label{fig:4}
\end{figure}

Next, let us briefly compare the computational efficiency of the three methods
under consideration. By also taking the necessity of drawing a random number at
each step into consideration, the two randomized Runge-Kutta methods
\eqref{eq5:NumMeth1} and \eqref{eq5:NumMeth2} are of course computationally
more expensive than the classical Euler method. For this reason we compare in
Figure~\ref{fig:4} 
the average CPU times of these three schemes versus their accuracy. From this
figure we can see that the classical Euler method is as expected the fastest
method and, since it still converges  with the same experimental order as the 
randomized Euler method \eqref{eq5:NumMeth1}, it is in total
more efficient than its randomized counter-part. On the other hand, the 
computationally even more expensive randomized Runge-Kutta method
\eqref{eq5:NumMeth2} quickly offsets its higher cost with its higher order of 
convergence.  

\section*{Acknowledgement}

The authors like to thank Wolf-J\"urgen Beyn, Monika Eisenmann, Mih\'aly
Kov\'acs, and Stig Larsson for inspiring discussions and helpful comments. This
research was carried out in the framework of \textsc{Matheon} supported by   
Einstein Foundation Berlin. The authors also gratefully acknowledge financial
support by the German Research Foundation through the research unit FOR 2402
-- Rough paths, stochastic partial differential equations and related topics --
at TU Berlin.

\end{document}